\documentclass[12pt]{amsart}
\usepackage{ifthen,verbatim}
\usepackage{floatflt,graphicx,graphics}
\usepackage{tikz}
\usepackage{mathrsfs}
\usepackage{amsmath,amsfonts,amssymb,amsthm}
\usepackage{subcaption}

\nonstopmode
\numberwithin{equation}{section}
\theoremstyle{plain}
\newtheorem{theorem}[equation]{Theorem}
\newtheorem{conjecture}[equation]{Conjecture}
\newtheorem{lemma}[equation]{Lemma}
\newtheorem{corollary}[equation]{Corollary}

\theoremstyle{definition}
\newtheorem{definition}[equation]{Definition}

\theoremstyle{remark}

\newcommand{\R}{\mathbb{R}}

\newcommand{\C}{\mathbb{C}}

\newcommand{\B}{\mathbb{B}}

\newcommand{\K}{\mathcal{K}}
\newcommand{\uhp}{\mathbb{H}}

\font\fFt=eusm10 
\font\fFa=eusm7  
\font\fFp=eusm5  
\def\K{\mathchoice{\hbox{\,\fFt K}}{\hbox{\,\fFt K}}{\hbox{\,\fFa K}}{\hbox{\,\fFp K}}}

\newcounter{alphabet}

\newcounter{minutes}
\divide\time by 60
\newcounter{hours}
\multiply\time by 60
\addtocounter{minutes}{-\time}

\begin{document}
\bibliographystyle{amsplain}
\title
{Inequalities for the generalized point pair function}

\def\thefootnote{}
\footnotetext{
\texttt{\tiny File:~\jobname .tex,
          printed: \number\year-\number\month-\number\day,
          \thehours.\ifnum\theminutes<10{0}\fi\theminutes}
}
\makeatletter\def\thefootnote{\@arabic\c@footnote}\makeatother

\author[O. Rainio]{Oona Rainio}

\begin{abstract}
We study a new generalized version of the point pair function defined with a constant $\alpha>0$. We prove that this function is a quasi-metric for all values of $\alpha>0$, and compare it to several hyperbolic-type metrics, such as the $j^*$-metric, the triangular ratio metric, and the hyperbolic metric. Most of the inequalities presented here have the best possible constants in terms of $\alpha$. Furthermore, we research the distortion of the generalized point pair function under conformal and quasiregular mappings.
\end{abstract}
\maketitle

\noindent Oona Rainio$^1$, email: \texttt{ormrai@utu.fi}, ORCID: 0000-0002-7775-7656,\newline 
1: University of Turku, FI-20014 Turku, Finland\\
\textbf{Funding.} My research was funded by Finnish Culture Foundation.\\
\textbf{Acknowledgements.} I am thankful for Professor Matti Vuorinen for his constructive comments related to this work and all his support.\\
\textbf{Data availability statement.} Not applicable, no new data was generated.\\
\textbf{Conflict of interest statement.} There is no conflict of interest.


\section{Introduction}

Several different metrics can be used to study conformal \cite{i13}, quasiconformal \cite{g06}, quasiregular, or other mappings \cite{o09}, which are an important subject of study in the geometric function theory. In the plane, the hyperbolic metric is useful for this purpose because of its invariance properties but, unfortunately, it can be defined only in special cases in dimensions $n\geq3$. For this reason, researchers have introduced numerous new hyperbolic-type metrics \cite{fmv,h,inm}, which are designed after the hyperbolic metric so that they can measure the distances between points by taking their location with respect to the domain boundary into account.

Let $G\subsetneq\R^n$ be a domain. For all points $x\in G$, denote the Euclidean distance to the boundary by $d_G(x)=\inf_{z\in\partial G}|x-z|$. For $\alpha>0$, define then the function $p_G^\alpha:G\times G\to[0,1)$, \cite[(5.1), p. 1391]{dnrv}
\begin{align}\label{def_gpf}
p^\alpha_G(x,y)=\frac{|x-y|}{\sqrt{|x-y|^2+\alpha d_G(x)d_G(y)}}.
\end{align}

This function above is called the \emph{generalized point pair function}. It is derived from the point pair function $p_G$, whose expression coincides with the special case $\alpha=4$ of the definition \eqref{def_gpf}. The point pair function was originally introduced in \cite{chkv} and studied further in \cite{dnrv,hvz,sch,fss,sqm}. It was observed to be a useful tool for creating bounds for the hyperbolic metric and proved to be a quasi-metric for all domains $G\subsetneq\R^n$ with the constant less than or equal to $\sqrt{5}/2$ \cite[Thm 4.14, p. 1388]{dnrv}. In 2022, Dautova et al. \cite{dnrv} introduced the generalized point pair function by replacing the constant 4 in the definition of the point pair function by a more general constant $\alpha>0$. 

Considering the point pair function is very well-justified if the domain $G$ is the upper half-space $\uhp^n=\{x=(x_1,...,x_n)\in\R^n\,|\,x_n>0\}$. This is because, for all points $x,y\in\uhp^n$, the distance $p_{\uhp^n}(x,y)$ in the point pair function is equal to the distance ${\rm th}(\rho_{\uhp^n}(x,y)/2)$, where th is the hyperbolic tangent and $\rho_{\uhp^n}(x,y)$ is the hyperbolic metric defined in $\uhp^n$. However, defining this function with another constant instead of 4 might be more reasonable in some domains, which is why studying the generalized point pair function for values of $\alpha>0$ is useful.  For instance, if $0<\alpha\leq12$, it is known that the generalized point pair function is a metric in the domains $\R^+$ \cite[Thm 5.2, p 1391]{dnrv}, $\R^n\setminus\{0\}$ \cite[Thm 5.11, p. 1395]{dnrv}, and $\uhp^n$ \cite[Thm 5.13, p. 1396]{dnrv}. Consequently, our aim is study this function further by comparing it to several hyperbolic-type metrics.

The structure of this article is as follows. First, in Section 2, we give the necessary notations and definitions. In Section 3, we study the inequality between the generalized point pair function and the hyperbolic-type metric known as the $j^*$-metric, and prove that the generalized point pair function is a quasi-metric in every domain $G\subsetneq\R^n$. In Section 4, we give the similar inequalities for the triangular ratio metric and the $t$-metric. Finally, in Section 5, we present the inequalities between the generalized point pair function and the hyperbolic metric and use them to find some results for the distortion of the distances in the generalized point pair function under conformal and quasiregular mappings. 

\section{Preliminaries}

Recall that a function $d:G\times G\to\R$ is a metric in a domain $G$ if, for all $x,y,z\in G$, (1) $d(x,y)\geq0$ and $d(x,y)=0$ if and only if $x=y$, (2) $d(x,y)=d(y,x)$, and (3) $d(x,y)\leq d(x,z)+d(z,y)$. The third one of these properties is called the \emph{triangle inequality}. If a function fulfills the two first properties and the relaxed version
\begin{align*}
d(x,y)\leq c(d(x,z)+d(z,y))    
\end{align*}
of the triangle inequality with a constant $c$ independent of the choice of the points $x,y,z$, then we call it a \emph{quasi-metric}. Note that this type of a function is sometimes referred as a semi-metric, a metametric, or an inframetric instead.

The Euclidean open ball with a center $x\in\R^n$ and a radius $r>0$ is denoted as $B^n(x,r)$ and its sphere is $S^{n-1}(x,r)$. A Euclidean line segment with endpoints $x,y\in\R^n$ is $[x,y]$. The argument of a complex number $x\in\C\setminus\{0\}$ is ${\rm Arg}(x)$. Denote also $d_G(x)=\inf_{z\in\partial G}|x-z|$ for $x\in\R^n$ as in Introduction.

Define then the original \emph{point pair function} \cite[p. 685]{chkv}, \cite[2.4, p. 1124]{hvz}, $p_G:G\times G\to[0,1),$
\begin{align*}
p_G(x,y)=\frac{|x-y|}{\sqrt{|x-y|^2+4d_G(x)d_G(y)}}.   
\end{align*}
The generalized point pair function is as in \eqref{def_gpf}. To avoid possible confusion, note that $p_G$ means that $\alpha=4$ and, if $\alpha$ is unspecified, we mean the generalized version $p^\alpha_G$.

Next, consider the following hyperbolic-type metrics. The \emph{distance ratio metric} introduced by Gehring and Osgood \cite{GO79} is defined as $j_G:G\times G\to[0,\infty)$, \cite[p. 685]{chkv}
\begin{align*}
j_G(x,y)=\log\left(1+\frac{|x-y|}{\min\{d_G(x),d_G(y)\}}\right).  \end{align*}
This expression can be modified as in \cite[2.2, p. 1123 \& Lemma 2.1, p. 1124]{hvz} to define the \emph{$j^*$-metric} $j^*_G:G\times G\to[0,1],$
\begin{align*}
j^*_G(x,y)={\rm th}\frac{j_G(x,y)}{2}=\frac{|x-y|}{|x-y|+2\min\{d_G(x),d_G(y)\}}.    
\end{align*}
The \emph{triangular ratio metric}, originally introduced by P. H\"ast\"o in 2002 \cite{h}, $s_G:G\times G\to[0,1],$ is defined as \cite[(1.1), p. 683]{chkv}  
\begin{align*}
s_G(x,y)=\frac{|x-y|}{\inf_{z\in\partial G}(|x-z|+|z-y|)}. 
\end{align*}
Furthermore, the \emph{$t$-metric} defined as $t_G:G\times G\to[0,1)$,
\begin{align*}
t_G(x,y)=\frac{|x-y|}{|x-y|+d_G(x)+d_G(y)},  
\end{align*}
was recently introduced in \cite{inm} but it must be noted that, unlike the distance ratio metric or the triangular ratio metric, this metric is not necessarily hyperbolic-type metric because the closures of its balls are not always compact in the domain $G$.

Use notations sh, ch and th for the hyperbolic sine, cosine, and tangent. Denote the upper half-space $\{x=(x_1,...,x_n)\in\R^n\,|\,x_n>0\}$ by $\uhp^n$ and use the notation $\B^n$ for the Poincar\'e unit ball $\{x\in\R^n\,|\,|x|<1\}$. In these two domains, the hyperbolic metric has the following formulas \cite[(4.8), p. 52 \& (4.14), p. 55]{hkv}
\begin{align*}
\text{ch}\rho_{\uhp^n}(x,y)&=1+\frac{|x-y|^2}{2d_{\uhp^n}(x)d_{\uhp^n}(y)},\quad x,y\in\uhp^n,\\
\text{sh}^2\frac{\rho_{\B^n}(x,y)}{2}&=\frac{|x-y|^2}{(1-|x|^2)(1-|y|^2)},\quad x,y\in\B^n.
\end{align*}
In the two-dimensional disk, we have
\begin{align*}
\text{th}\frac{\rho_{\B^2}(x,y)}{2}
=\left|\frac{x-y}{1-x\overline{y}}\right|,
\end{align*}
where $\overline{y}$ is the complex conjugate of $y$.

\section{Inequalities with the $j^*$-metric}

In this section, we first find the inequalities between the generalized point pair function and the $j^*$-metric, then study the sharpness of the established inequalities, and use them to prove that the generalized point pair function is a quasi-metric.

\begin{theorem}\label{thm_jpG}
For all $x,y\in G\subsetneq\R^n$ and $\alpha>0$, the inequality
\begin{align*}
\min\left\{1,\frac{2}{\sqrt{\alpha}}\right\}j^*_G(x,y)\leq p^\alpha_G(x,y)\leq\sqrt{\frac{\alpha+4}{\alpha}}j^*_G(x,y)
\end{align*}
holds. For the domain $G=\uhp^n$, these constants are the best ones possible in terms of $\alpha$. In fact, the first constant here is the best one possible in terms of $\alpha$ for every choice of the domain $G\subsetneq\R^n$. 
\end{theorem}
\begin{proof}
By symmetry, we can fix distinct points $x,y\in G$ such that $d_G(x)\leq d_G(y)$. Now,
\begin{align}\label{quo_jgp}
\frac{j^*_G(x,y)}{p^\alpha_G(x,y)}=\frac{\sqrt{|x-y|^2+\alpha d_G(x)d_G(y)}}{|x-y|+2d_G(x)}.
\end{align}
Clearly, for fixed choices of $d_G(x)$ and $|x-y|$, this quotient is increasing with respect to $d_G(y)$. Because of the triangle inequality, $d_G(y)\leq d_G(x)+|x-y|$, so the value of $d_G(y)$ is limited to the closed interval from $d_G(x)$ to $d_G(x)+|x-y|$. Consequently, the quotient \eqref{quo_jgp} is at minimum with respect to $d_G(y)$ when $d_G(y)=d_G(x)$ and at maximum when $d_G(y)=d_G(x)+|x-y|$.

Let us first find the minimum of the quotient \eqref{quo_jgp} in the case $d_G(y)=d_G(x)$. By differentiation,
\begin{align*}
&\frac{\partial}{\partial|x-y|}\left(\frac{\sqrt{|x-y|^2+\alpha d_G(x)^2}}{|x-y|+2d_G(x)}\right) 
=\frac{d_G(x)(2|x-y|-\alpha d_G(x))}{\sqrt{|x-y|^2+\alpha d_G(x)^2}(|x-y|+2d_G(x))^2}=0\\
&\Leftrightarrow\quad |x-y|=\frac{\alpha}{2}d_G(x).
\end{align*}
We see that the stationary point above is a minimum and, since this quotient is $\sqrt{\alpha/(\alpha+4)}$ at $|x-y|=\alpha d_G(x)/2$, this is the minimum of the quotient \eqref{quo_jgp}.

As explained above, we need to fix $d_G(y)=d_G(x)+|x-y|$ to find the maximum value of the quotient \eqref{quo_jgp}. By differentiation,
\begin{align*}
&\frac{\partial}{\partial|x-y|}\left(\frac{\sqrt{|x-y|^2+\alpha d_G(x)(d_G(x)+|x-y|)}}{|x-y|+2d_G(x)}\right)\\ 
&=\frac{d_G(x)|x-y|(4-\alpha)}{2\sqrt{|x-y|^2+\alpha d_G(x)(d_G(x)+|x-y|)}(|x-y|+2d_G(x))^2}\geq0
\quad\Leftrightarrow\quad\alpha\leq4.
\end{align*}
Because this quotient is either decreasing or increasing with respect to $|x-y|$ depending on $\alpha$, its maximum has either one of the limit values:
\begin{align*}
&\lim_{|x-y|\to0^+}\frac{\sqrt{|x-y|^2+\alpha d_G(x)(d_G(x)+|x-y|)}}{|x-y|+2d_G(x)}=\frac{\sqrt{\alpha}}{2},\\
&\lim_{|x-y|\to\infty}\frac{\sqrt{|x-y|^2+\alpha d_G(x)(d_G(x)+|x-y|)}}{|x-y|+2d_G(x)}=1.
\end{align*}
Consequently, the supremum of the quotient \eqref{quo_jgp} is $\max\{1,\sqrt{\alpha}/2\}$.

Consider then the domain $G=\uhp^n$. The quotient \eqref{quo_jgp} attains its minimum $\sqrt{\alpha/(\alpha+4)}$ when $x=(0,...,0,1)$ and $y=(\alpha/2,0,...,0,1)$ as in Figure 1(A). Similarly, it approaches its maximum value $\max\{1,\sqrt{\alpha}/2\}$ when $x=(0,...,0,1)$ and $y=(0,...,0,1+k)$ with $k\to0^+$ if $\alpha<4$ or $k\to\infty$ if $\alpha\geq4$. Consequently, $\uhp^n$ is an example of a domain in which the found extreme values offer the best possible constants in terms of $\alpha$. Also, $\max\{1,\sqrt{\alpha}/2\}$ is the best possible upper bound for the quotient \eqref{quo_jgp} in terms of $\alpha$, regardless of how the domain $G$ is chosen, because we can always fix points $y\in G$, $z\in S^{n-1}(y,d_G(y))\cap(\partial G)$, and $x=y+k(z-y)$ with either $k\to0^+$ or $k\to1^-$, depending on $\alpha$, so that we attain the limit value of the quotient \eqref{quo_jgp}. 

The theorem follows from this, though note that we consider the reciprocals of the found extreme values since the bounds are presented for the function $p^\alpha_G(x,y)$.
\end{proof}

While the latter constant in Theorem \ref{thm_jpG} is not sharp for some choices of $G\subsetneq\R^n$, it follows from the next result that this constant is the best possible one in several common domains such as $\B^n$, $\uhp^n$, and $\R^n\setminus(\{0\}\cup\{1\})$.

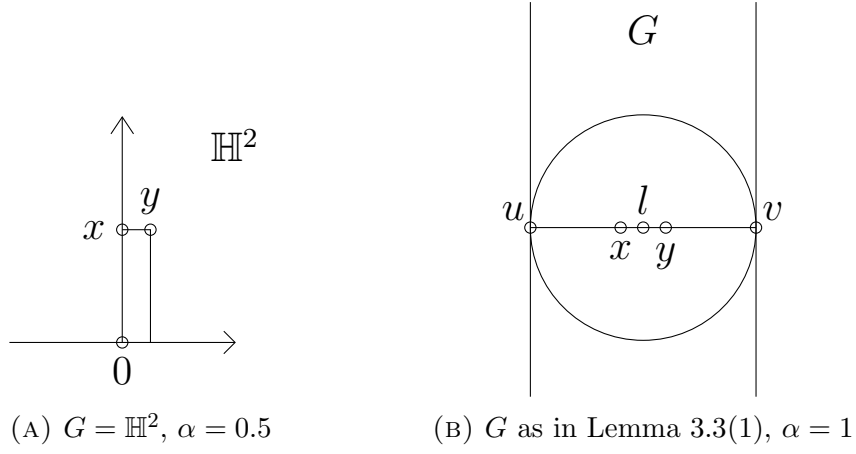
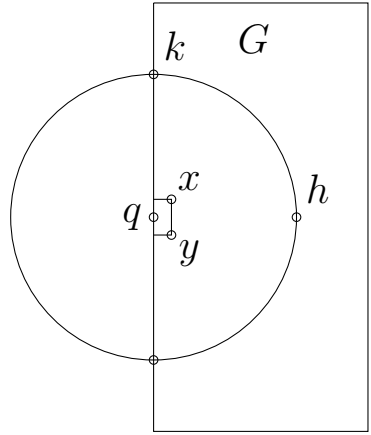
\begin{figure}[ht]
  \centering
  \begin{subfigure}[b]{0.35\textwidth}
  \centering
  \begin{tikzpicture}[scale=1.5]
    \node[scale=1.3] at (1,1.75) {$\uhp^2$};
    \draw (0,1) circle (0.05);
    \node[scale=1.3] at (-0.25,1) {$x$};
    \draw (0.25,1) circle (0.05);
    \node[scale=1.3] at (0.25,1.25) {$y$};
    \draw (-1,0) -- (1,0);
    \draw (0,0) circle (0.05);
    \node[scale=1.3] at (0,-0.25) {$0$};
    \draw (0,1) -- (0.25,1) -- (0.25,0);
    \draw (0,0) -- (0,2);
    \draw (-0.1,1.85) -- (0,2) -- (0.1,1.85);
    \draw (0.9,-0.1) -- (1,0) -- (0.9,0.1);
  \end{tikzpicture}
  \caption{$G=\uhp^2$, $\alpha=0.5$}
  \end{subfigure}
  \hspace{0cm}
  \begin{subfigure}[b]{0.45\textwidth}
  \centering
  \begin{tikzpicture}[scale=1.5]
    \draw (0,0) circle (0.05);
    \node[scale=1.3] at (0,0.25) {$l$};
    \draw (0,0) circle (1);
    \draw (1,2) -- (1,-1.5);
    \draw (-1,2) -- (-1,-1.5);
    \node[scale=1.3] at (0,1.75) {$G$};
    \draw (-1,0) circle (0.05);
    \node[scale=1.3] at (-1.15,0.15) {$u$};
    \draw (1,0) circle (0.05);
    \node[scale=1.3] at (1.15,0.15) {$v$};
    \draw (-0.2,0) circle (0.05);
    \node[scale=1.3] at (-0.2,-0.2) {$x$};
    \draw (0.2,0) circle (0.05);
    \node[scale=1.3] at (0.2,-0.25) {$y$};
    \draw (-1,0) -- (1,0);
    \end{tikzpicture}
  \caption{$G$ as in Lemma \ref{lem_shr}(1), $\alpha=1$}
  \end{subfigure}
  \\
  \vspace{1cm}
  \hspace{1.55cm}
  \begin{subfigure}[b]{0.5\textwidth}
  \centering
  \begin{tikzpicture}[scale=1.9]
    \draw (0,0) circle (0.03);
    \node[scale=1.3] at (-0.15,0) {$q$};
    \draw (0,0) circle (1);
    \draw (0,1.5) -- (0,-1.5) -- (1.5,-1.5) -- (1.5,1.5) -- (0,1.5);
    \node[scale=1.3] at (0.7,1.25) {$G$};
    \draw (1,0) circle (0.03);
    \node[scale=1.3] at (1.15,0.2) {$h$};
    \draw (0,1) circle (0.03);
    \node[scale=1.3] at (0.15,1.2) {$k$};
    \draw (0.125,0.125) circle (0.03);
    \node[scale=1.3] at (0.25,0.25) {$x$};
    \draw (0.125,-0.125) circle (0.03);
    \node[scale=1.3] at (0.25,-0.25) {$y$};
    \draw (0,0.125) -- (0.125,0.125) -- (0.125,-0.125) -- (0,-0.125);
    \draw (0,-1) circle (0.03);
    \end{tikzpicture}
  \caption{$G$ as in Lemma \ref{lem_shr}(2), $\alpha=4$}
  \end{subfigure}
  \caption{Points $x,y\in G$ such that the equality $p^\alpha_G(x,y)=\sqrt{(\alpha+4)/\alpha}j^*_G(x,y)$ holds for different domains $G$ and values of $\alpha>0$.}
  \label{fig3}
\end{figure}

\begin{lemma}\label{lem_shr}
For $G\subsetneq\R^n$ and $\alpha>0$, the constant $\sqrt{(\alpha+4)/\alpha}$ is the best possible constant $c$ in terms of $\alpha$ such that the inequality $p^\alpha_G(x,y)\leq cj^*_G(x,y)$ holds for all points $x,y\in G$ if \newline(1) $G$ contains an open ball so that the end points of one of its diameters belong to the boundary $\partial G$, or, \newline(2) $G$ contains an open half-ball but one of its diameters is fully on the boundary $\partial G$.
\end{lemma}
\begin{proof}
The inequality holds with $c=\sqrt{(\alpha+4)/\alpha}$ according to Theorem \ref{thm_jpG} and it can be trivially verified that the equality holds here if $d_G(x)=d_G(y)$ and $|x-y|=\alpha d_G(x)/2$.

Consider the first case where there are some points $u,v\in S^1(l,d_G(l))\cap(\partial G)$ for $l=(u+v)/2$. Fix then $x=l+\alpha(u-l)/(\alpha+4)$ and $y=l+\alpha(v-l)/(\alpha+4)$. See Figure 1(B). We will have
$d_G(x)=4d_G(z)/(\alpha+4)=d_G(y)$, and $|x-y|=\alpha d_G(x)/2$.

Suppose then that for $q\in\partial G$, $r>0$, and $h\in S^{n-1}(q,r)$, the half-ball
\begin{align*}
\{z\in B^n(q,r)\,:\,|z-h|>|z-(h+2r(q-h))|\}    
\end{align*}
is included in $G$ but 
\begin{align*}
[k,k+2(q-k)]\subset\partial G\quad\text{for}\quad k\in\{z\in S^{n-1}(q,r)\,:\,|z-h|=|z-(h+2r(q-h))|\}.    
\end{align*}
Fix 
\begin{align*}
x=q+\frac{h-q}{\alpha+4}+\alpha\frac{k-q}{4(\alpha+4)},\quad 
y=q+\frac{h-q}{\alpha+4}-\alpha\frac{k-q}{4(\alpha+4)}
\end{align*}
as in Figure 1(C). Now,
$d_G(x)=d_G(y)=r/(\alpha+4)$ and
$|x-y|=\alpha r/(2(\alpha+4))=\alpha d_G(x)/2$, so the result follows.
\end{proof}

\begin{lemma}\label{lem_rminus0}
For all $x,y\in\R^n\setminus\{0\}$ and $\alpha>0$, the inequalities
\begin{align*}
\frac{2}{\sqrt{\alpha}}j^*_{\R^n\setminus\{0\}}(x,y)&\leq p^\alpha_{\R^n\setminus\{0\}}(x,y)\leq\sqrt{1+\frac{4}{\alpha}}j^*_{\R^n\setminus\{0\}}(x,y)\quad\text{if}\quad\alpha\leq4,\\
j^*_{\R^n\setminus\{0\}}(x,y)&\leq p^\alpha_{\R^n\setminus\{0\}}(x,y)\leq\max\left\{1,\frac{4}{\sqrt{\alpha+4}}\right\}j^*_{\R^n\setminus\{0\}}(x,y)\quad\text{if}\quad\alpha>4,
\end{align*}
hold with the best possible constants in terms of $\alpha$.
\end{lemma}
\begin{proof}
The left sides of both of the inequalities follow from Theorem \ref{thm_jpG}, according to which they also have the best possible constants. By symmetry, assume that $|x|\leq|y|$ for the points $x,y\in G=\R^n\setminus\{0\}$. Let $k$ be the angle between the vectors from the origin to $x$ and to $y$. By writing the distance $|x-y|$ with law of cosines, we will have
\begin{align}\label{quo_pjR0}
\frac{p^\alpha_{\R^n\setminus\{0\}}(x,y)}{j^*_{\R^n\setminus\{0\}}(x,y)}=\frac{\sqrt{|x|^2+|y|^2-2|x||y|\cos(k)}+2|x|}{\sqrt{|x|^2+|y|^2-2|x||y|\cos(k)+\alpha|x||y|}}.    
\end{align}
To prove the right side of the inequalities in the lemma, we need to find the maximum value of this quotient.

By differentiation,
\begin{align*}
&\frac{\partial}{\partial\cos(k)}\left(\frac{\sqrt{|x|^2+|y|^2-2|x||y|\cos(k)}+2|x|}{\sqrt{|x|^2+|y|^2+(\alpha-2\cos(k))|x||y|}}\right)\\ 
&=\frac{|x|^2|y|(2\sqrt{|x|^2+|y|^2-2|x||y|\cos(k)}-\alpha|y|)}{\sqrt{|x|^2+|y|^2-2|x||y|\cos(k)}(|x|^2+|y|^2+(\alpha-2\cos(k))|x||y|)^{3/2}}=0\\
&\Leftrightarrow\quad
\cos(k)=\frac{4|x|^2+4|y|^2-\alpha^2|y|^2}{8|x||y|}.
\end{align*}
This stationary point is a maximum. It fulfills $-1\leq\cos(k)\leq1$ if and only if
\begin{align*}
-4(|y|-|x|)^2\leq\alpha^2|y|^2\leq4(|x|+|y|)^2
\quad\Leftrightarrow\quad
(\alpha-2)|y|\leq2|x|,
\end{align*}
which is only possible for $\alpha\leq4$ given the limitation $|x|\leq|y|$. If $\cos(k)$ is equivalent to the stationary point above, the quotient \eqref{quo_pjR0} becomes $\sqrt{1+4|x|/(\alpha|y|)}$, which is decreasing with respect to $|y|$ and attains its maximum value $\sqrt{1+4/\alpha}$ at $|y|=|x|$. Consequently, this is the maximum value of the quotient \eqref{quo_pjR0} if $\alpha\leq4$. 

Suppose then that $\alpha>4$. Now, we must choose either $\cos(k)=-1$ or $\cos(k)=1$ to find the maximum value of the quotient \eqref{quo_pjR0}. If $\cos(k)=-1$, this quotient becomes
\begin{align}\label{quo_pjR0_0}
\frac{|y|+3|x|}{\sqrt{|x|^2+|y|^2+(\alpha+2)|x||y|}}.    
\end{align}
By differentiation,
\begin{align*}
\frac{\partial}{\partial|y|}\left(\frac{|y|+3|x|}{\sqrt{|x|^2+|y|^2+(\alpha+2)|x||y|}}\right)=\frac{-|x|((2+\alpha)|x|+|y|/2)}{(|x|^2+|y|^2+(\alpha+2)|x||y|)^{3/2}}<0,    
\end{align*}
so the quotient \eqref{quo_pjR0_0} is decreasing with respect to $|y|$ and has a maximum value $4/\sqrt{\alpha+4}$ at $|y|=|x|$.

If $\cos(k)=1$ instead, the quotient \eqref{quo_pjR0} becomes
\begin{align}\label{quo_pjR0_1}
\frac{|x|+|y|}{\sqrt{|x|^2+|y|^2+(\alpha-2)|x||y|}}.    
\end{align}
By differentiation,
\begin{align*}
\frac{\partial}{\partial|y|}\left(\frac{|x|+|y|}{\sqrt{|x|^2+|y|^2+(\alpha-2)|x||y|}}\right)=\frac{|x|(|y|-|x|)(\alpha/2-2)}{(|x|^2+|y|^2+(\alpha-2)|x||y|)^{3/2}}\geq0    
\end{align*}
if $\alpha>4$ and $|x|\leq|y|$. Consequently, the quotient \eqref{quo_pjR0_1} is increasing with respect to $|y|$ and its maximum has a limit value 1 obtained when $|y|\to\infty$. Thus, the result follows.
\end{proof}

\begin{corollary}\label{cor_quasi}
For all $G\subsetneq\R^n$ and $\alpha>0$, the function $p^\alpha_G(x,y)$ is a quasi-metric with a constant $\sqrt{(\alpha+4)/\alpha}$ if $\alpha\leq4$ and $2\sqrt{\alpha+4}/\alpha$ if $\alpha>4$.
\end{corollary}
\begin{proof}
It follows from Theorem \ref{thm_jpG} and the fact that $j^*_G(x,y)$ is a metric that
\begin{align*}
p^\alpha_G(x,y)\leq&\sqrt{\frac{\alpha+4}{\alpha}}j^*_G(x,y)\leq\sqrt{\frac{\alpha+4}{\alpha}}(j^*_G(x,z)+j^*_G(z,y))\\
&\leq\sqrt{\frac{\alpha+4}{\alpha}}\min\left\{1,\frac{2}{\sqrt{\alpha}}\right\}(p^\alpha_G(x,z)+p^\alpha_G(z,y)).  
\end{align*}
\end{proof}

Note that, for the domain $G=\R^n\setminus\{0\}$, the inequalities of Lemma \ref{lem_rminus0} would give better constants for Corollary \ref{cor_quasi}, but it is already proven that the generalized point pair function is a metric in $\R^n\setminus\{0\}$ for all $0<\alpha\leq12$ \cite[Thm 5.11, p. 1395]{dnrv}.

\section{Inequalities with the triangular ratio metric and the $t$-metric}

Let us now find the inequalities between the generalized point pair function and the triangular ratio metric and the $t$-metric.

\begin{lemma}\label{lem_pbing}
For all $x,y\in G\subsetneq\R^n$ and $\beta>\alpha>0$, the inequality
\begin{align*}
p^\beta_G(x,y)\leq p^\alpha_G(x,y)\leq\sqrt{\frac{\beta}{\alpha}}p^\beta_G(x,y)    
\end{align*}
holds with the best possible constants in terms of $\alpha$.
\end{lemma}
\begin{proof}
Clearly, the quotient
\begin{align}
\frac{p^\alpha_G(x,y)}{p^\beta_G(x,y)}=\sqrt{\frac{|x-y|^2+\beta d_G(x)d_G(y)}{|x-y|^2+\alpha d_G(x)d_G(y)}}.    
\end{align}
attains its minimum value 1 when either $x$ or $y$ approaches boundary so that $d_G(x)\to0^+$ or $d_G(y)\to0^+$, and its maximum value $\sqrt{\beta/\alpha}$ when the points $x$ and $y$ approach to each other so that $|x-y|\to0^+$.
\end{proof}

\begin{lemma}
For all $x,y\in G\subsetneq\R^n$ and $\alpha>0$,
\begin{align*}
\frac{1}{2}s_G(x,y)&\leq p^\alpha_G(x,y)\leq\sqrt{\frac{\alpha+4}{\alpha}}s_G(x,y)\quad\text{if}\quad\alpha\leq4,\\
\frac{1}{\sqrt{2}}s_G(x,y&)\leq p^\alpha_G(x,y)\leq\sqrt{\frac{\alpha+4}{\alpha}}s_G(x,y)\quad\text{if}\quad\alpha>4,
\end{align*}
and, if $G$ is convex, the left sides of these inequalities can be improved by replacing constants $1/2$ and $1/\sqrt{2}$ by $\max\{1/\sqrt{2},\sqrt{\alpha}/2\}$ and $1$, respectively.
\end{lemma}
\begin{proof}
By \cite[Lemma 2.1, p. 1124 \& Lemma 2.2, p. 1125]{hvz}, the inequality $j^*_G(x,y)\leq s_G(x,y)\leq2j^*_G(x,y)$ holds for all $x,y\in G\subsetneq\R^n$, and, by combining this to Theorem \ref{thm_jpG}, we will have
\begin{align}\label{ine_sp0}
\min\left\{\frac{1}{2},\frac{1}{\sqrt{\alpha}}\right\}s_G(x,y)\leq p^\alpha_G(x,y)\leq\sqrt{\frac{\alpha+4}{\alpha}}s_G(x,y).  \end{align}
It also follows from Lemma \ref{lem_pbing} that
\begin{align}\label{ine_pfour}
\min\left\{1,\frac{\sqrt{\alpha}}{2}\right\}p_G(x,y)\leq p^\alpha_G(x,y)\leq\max\left\{1,\frac{\sqrt{\alpha}}{2}\right\}p_G(x,y)   
\end{align}
and, by \cite[Thm 3.6]{sqm}, $1/\sqrt{2}p_G(x,y)\leq s_G(x,y)\leq\sqrt{2}p_G(x,y)$. Consequently, we will have
\begin{align}\label{ine_sp1}
\frac{1}{\sqrt{2}}\min\left\{1,\frac{\sqrt{\alpha}}{2}\right\}s_G(x,y)\leq p^\alpha_G(x,y)\leq\max\left\{\sqrt{2},\sqrt{\frac{\alpha}{2}}\right\}s_G(x,y).   
\end{align}

Let us now combine the inequalities \eqref{ine_sp0} and \eqref{ine_sp1}. Note that
\begin{align*}
&\max\left\{\min\left\{\frac{1}{2},\frac{1}{\sqrt{\alpha}}\right\},\min\left\{\frac{1}{\sqrt{2}},\frac{\sqrt{\alpha}}{2\sqrt{2}}\right\}\right\}
=\begin{cases}
1/2 &\text{if}\quad\alpha\leq4,\\
1/\sqrt{2} &\text{if}\quad\alpha>4
\end{cases},\quad\text{and}\\
&\sqrt{\frac{\alpha+4}{\alpha}}\leq\max\left\{\sqrt{2},\sqrt{\frac{\alpha}{2}}\right\}.
\end{align*}
The first part of the lemma follows from this.

Suppose then that $G$ is convex. By \cite[Thm 2.9(i), p. 1129]{hvz}, $s_G(x,y)\leq\sqrt{2}j^*_G(x,y)$ holds in this case so it follows from Theorem \ref{thm_jpG} that
\begin{align*}
\min\left\{\frac{1}{\sqrt{2}},\sqrt{\frac{2}{\alpha}}\right\}s_G(x,y)\leq p^\alpha_G(x,y)   
\end{align*}
Furthermore, $s_G(x,y)\leq p_G(x,y)$ in a convex domain $G$ by \cite[lemma 11.6(1), p. 197]{hkv}, so it follows from the inequality \eqref{ine_pfour} that
\begin{align*}
\min\left\{1,\frac{\sqrt{\alpha}}{2}\right\}s_G(x,y)\leq p^\alpha_G(x,y).    
\end{align*}
The rest of the lemma follows from the two inequalities above as
\begin{align*}
\max\left\{\min\left\{\frac{1}{\sqrt{2}},\sqrt{\frac{2}{\alpha}}\right\},\min\left\{1,\frac{\sqrt{\alpha}}{2}\right\}\right\}=\begin{cases}
1/\sqrt{2} &\text{if}\quad\alpha\leq2,\\
\sqrt{\alpha}/2 &\text{if}\quad2<\alpha\leq4,\\
1 &\text{if}\quad\alpha>4.
\end{cases}    
\end{align*}
\end{proof}

\begin{lemma}
For all $x,y\in G\subsetneq\R^n$ and $\alpha>0$, the inequalities
\begin{align*}
t_G(x,y)&\leq p^\alpha_G(x,y)\leq\frac{4}{\sqrt{\alpha(4-\alpha)}}t_G(x,y)\quad\text{if}\quad\alpha<2,\\
\min\left\{1,\frac{2}{\sqrt{\alpha}}\right\}t_G(x,y)&\leq p^\alpha_G(x,y)\leq2t_G(x,y)\quad\text{if}\quad\alpha\geq2,
\end{align*}
hold with the best possible constant in terms of $\alpha$.
\end{lemma}
\begin{proof}
Consider the quotient
\begin{align}\label{quo_tpg}
\frac{p^\alpha_G(x,y)}{t_G(x,y)}=\frac{|x-y|+d_G(x)+d_G(y)}{\sqrt{|x-y|^2+\alpha d_G(x)d_G(y)}}.    
\end{align}
By differentiation,
\begin{align*}
&\frac{\partial}{\partial d_G(y)}\left(\frac{|x-y|+d_G(x)+d_G(y)}{\sqrt{|x-y|^2+\alpha d_G(x)d_G(y)}}\right)=\frac{|x-y|^2+\dfrac{\alpha}{2}d_G(x)(d_G(y)-|x-y|-d_G(x))}{2(|x-y|^2+\alpha d_G(x)d_G(y))^{3/2}}=0\\
&\Leftrightarrow\quad d_G(y)=|x-y|+d_G(x)-\frac{2|x-y|^2}{\alpha d_G(x)}.
\end{align*}
The stationary point above is a minimum. By symmetry, let us assume that $d_G(x)\leq d_G(y)$. It follows from the triangle inequality that $d_G(y)\leq |x-y|+d_G(x)$. Consequently,
we can choose $d_G(y)=|x-y|+d_G(x)-2|x-y|^2/(\alpha d_G(x))$ if and only if
\begin{align*}
d_G(x)\leq|x-y|+d_G(x)-\frac{2|x-y|^2}{\alpha d_G(x)}\leq|x-y|+d_G(x)\quad
\Leftrightarrow\quad 0\leq|x-y|\leq\frac{\alpha d_G(x)}{2}. \end{align*}

Suppose first that $|x-y|\leq\alpha d_G(x)/2$. If $d_G(y)=|x-y|+d_G(x)-2|x-y|^2/(\alpha d_G(x))$, the quotient \eqref{quo_tpg} becomes
\begin{align*}
\frac{2}{\alpha d_G(x)}\sqrt{\alpha d_G(x)(|x-y|+d_G(x))-|x-y|^2}.   \end{align*}
By differentiation,
\begin{align*}
\frac{\partial}{\partial |x-y|}\left(\alpha d_G(x)(|x-y|+d_G(x))-|x-y|^2\right)=\alpha d_G(x)-2|x-y|
\end{align*}
so the expression is increasing with respect to $|x-y|$ when $|x-y|\leq(\alpha d_G(x))/2$. It has a limit value
\begin{align*}
\lim_{|x-y|\to0^+}\left(\frac{2}{\alpha d_G(x)}\sqrt{\alpha d_G(x)(|x-y|+d_G(x))-|x-y|^2}\right)=\frac{2}{\sqrt{\alpha}}.    
\end{align*}

Consider then the case $|x-y|>\alpha d_G(x)/2$. Now, the quotient \eqref{quo_tpg} is increasing with respect to $d_G(y)$. If $d_G(y)=d_G(x)$, the quotient \eqref{quo_tpg} becomes
\begin{align}\label{quo_tpxy}
\frac{|x-y|+2d_G(x)}{\sqrt{|x-y|^2+\alpha d_G(x)^2}}.    
\end{align}
By differentiation,
\begin{align*}
\frac{\partial}{\partial |x-y|}\left(\frac{|x-y|+2d_G(x)}{\sqrt{|x-y|^2+\alpha d_G(x)^2}}\right)=\frac{d_G(x)(\alpha d_G(x)-2)}{(|x-y|^2+\alpha d_G(x)^2)^{3/2}}\geq0
\quad\Leftrightarrow\quad d_G(x)\geq\frac{2}{\alpha}.
\end{align*}
Consequently, the quotient \eqref{quo_tpg} is monotonic with respect to $|x-y|$. The quotient \eqref{quo_tpg} has a limit value $\sqrt{1+4/\alpha}$ when $|x-y|\to\alpha d_G(x)/2$ and a limit value 1 when $|x-y|\to\infty$, out of which 1 is smaller. Thus, it follows that the infimum of the quotient \eqref{quo_tpg} is $\min\{1,2/\sqrt{\alpha}\}$.

It follows from the earlier differentiation of the quotient \eqref{quo_tpg} that it is at maximum with respect to $d_G(y)$ in one of the end points of the interval $[d_G(x),d_G(x)+|x-y|]$. If $d_G(y)=d_G(x)$, the quotient \eqref{quo_tpg} is the quotient \eqref{quo_tpxy}, which was noted to be monotonic with respect to $|x-y|$. The maximum value of the quotient \eqref{quo_tpxy} has either a limit value $2/\sqrt{\alpha}$ obtained when $|x-y|\to0^+$ or 1 obtained when $|x-y|\to\infty$, depending if $\alpha\leq4$ or not.

If $d_G(y)=d_G(x)+|x-y|$, the quotient \eqref{quo_tpg} becomes
\begin{align}\label{quo_tpxy1}
\frac{2(|x-y|+d_G(x))}{\sqrt{|x-y|^2+\alpha d_G(x)(d_G(x)+|x-y|)}}.  \end{align}
By differentiation,
\begin{align*}
&\frac{\partial}{\partial |x-y|}\left(\frac{2(|x-y|+d_G(x))}{\sqrt{|x-y|^2+\alpha d_G(x)(d_G(x)+|x-y|)}}\right)\\
&=\frac{d_G(x)((\alpha-2)|x-y|+\alpha d_G(x))}{(|x-y|^2+\alpha d_G(x)(d_G(x)+|x-y|))^{3/2}}.    
\end{align*}
The derivative above is positive if either $\alpha\geq2$ or $\alpha<2$ and $|x-y|<\alpha d_G(x)/(2-\alpha)$. Consequently, if $\alpha\geq2$, the quotient \eqref{quo_tpxy1} is increasing with respect to $|x-y|$ and its maximum has a limit value 2 obtained when $|x-y|\to\infty$. If $\alpha<2$ instead, the maximum of the quotient \eqref{quo_tpxy1} is $4/\sqrt{\alpha(4-\alpha)}$ at $|x-y|=\alpha d_G(x)/(2-\alpha)$. Because these values are greater than the limit values of the maximum values of the quotient \eqref{quo_tpxy}, it follows that the supremum of the quotient \eqref{quo_tpg} is either $4/\sqrt{\alpha(4-\alpha)}$ if $\alpha<2$ and 2 if $\alpha\geq2$. The inequalities of the lemma now follow.   

The limit value $2/\sqrt{\alpha}$ of the quotient \eqref{quo_tpg} can be obtained in any domain $G\subsetneq\R^n$ by choosing $y\in B^n(x,d_G(x))$ for any fixed point $x\in G$ so that $d_G(y)\to d_G(x)$ and $|x-y|\to0^+$. Similarly, the limit value 1 can be found by choosing $x,y\in G$ so that $d_G(x),d_G(y)\to0^+$.
If $\alpha<2$, the value $4/\sqrt{\alpha(4-\alpha)}$ of the quotient \eqref{quo_tpg} is possible to find by fixing $x\in G$, $z\in S^{n-1}(x,d_G(x))\cap(\partial G)$, and $y=x+\alpha(z-x)/2$. Furthermore, the limit value 2 of the quotient \eqref{quo_tpg} can be obtained if we fix $x\in G$, $z\in S^{n-1}(x,d_G(x))\cap(\partial G)$ and $y=z+k(x-z)$ with $k\to0^+$ because then $d_G(y)\to0^+$ but $|x-y|\to d_G(x)>0$. It follows from this that we have the best constants in terms of $\alpha$, regardless of the choice of the domain $G$.
\end{proof}

\section{Inequalities with the hyperbolic metric}

In this section, we first study the inequalities between the generalized point pair function and the hyperbolic metric in the upper half-space and the unit ball, and then study the distortion of the generalized point pair function under M\"obius and quasiregular mappings. 

\begin{corollary}
For all $x,y\in\uhp^n$ and $\alpha>0$, the inequality
\begin{align*}
\min\left\{1,\frac{\sqrt{\alpha}}{2}\right\}{\rm th}\frac{\rho_{\uhp^n}(x,y)}{2}\leq p^\alpha_{\uhp^n}(x,y)\leq\max\left\{1,\frac{2}{\sqrt{\alpha}}\right\}{\rm th}\frac{\rho_{\uhp^n}(x,y)}{2}   
\end{align*}
holds with the best possible constants in terms of $\alpha$.
\end{corollary}
\begin{proof}
The results follows from Lemma \ref{lem_pbing} and the fact that ${\rm th}(\rho_{\uhp^n}(x,y)/2)=p_{\uhp^n}(x,y)$ by \cite[p. 460]{hkv}.    
\end{proof}

\begin{theorem}\label{thm_prhoB}
For all $x,y\in\B^n$ and $\alpha>0$, the inequality
\begin{align*}
\min\left\{1,\frac{1}{\sqrt{\alpha}}\right\}{\rm th}\frac{\rho_{\B^n}(x,y)}{2}\leq p^\alpha_{\B^n}(x,y)\leq\max\left\{1,\frac{2}{\sqrt{\alpha}}\right\}{\rm th}\frac{\rho_{\B^n}(x,y)}{2}
\end{align*}
holds with the best possible constants in terms of $\alpha$.
\end{theorem}
\begin{proof}
The values of $p^\alpha_{\B^n}(x,y)$ and $\rho_{\B^n}(x,y)$ only depend on how the points $x,y$ are fixed on the intersection of the unit ball and the two-dimensional plane containing $x$, $y$, and the origin, so it is enough to prove this inequality in case $n=2$ by studying the quotient
\begin{align}\label{quo_rpb}
\frac{p^\alpha_{\B^2}(x,y)}{{\rm th}(\rho_{\B^2}(x,y)/2)}=\frac{|1-x\overline{y}|}{\sqrt{|x-y|^2+\alpha(1-|x|)(1-|y|)}}.    
\end{align}

If $y=0$, the quotient \eqref{quo_rpb} becomes
\begin{align}\label{quo_rpby0}
\frac{1}{\sqrt{|x|^2+\alpha(1-|x|)}},    
\end{align}
which approaches $1/\sqrt{\alpha}$ when $|x|\to0^+$ and 1 when $|x|\to1^-$. By differentiation,
\begin{align*}
\frac{\partial}{\partial|x|}(|x|^2+\alpha(1-|x|))=2|x|-\alpha=0
\quad\Leftrightarrow\quad
|x|=\alpha/2.
\end{align*}
It follows that, if $\alpha<2$, the maximum of the quotient \eqref{quo_rpby0} is $2/\sqrt{\alpha(4-\alpha)}$. Otherwise, the maximum is 1. The minimum of the quotient \eqref{quo_rpby0} is also 1 or $1/\sqrt{\alpha}$, depending if $\alpha<1$ or not. By symmetry, these are the extreme values of the quotient \eqref{quo_rpb} also in the case $x=0$. 

Suppose then that $x\neq0\neq y$. Let $k\in[0,\pi]$ be the angle between the vectors from the origin to $x$ and $y$, or equivalently $k={\rm Arg}(x/y)$. By law of cosines,
\begin{align*}
|1-x\overline{y}|&=|1-|x||y|e^{{\rm Arg}(x/y)i}|=\sqrt{1+|x|^2|y|^2-2|x||y|\cos(k)},\\
|x-y|&=\sqrt{|x|^2+|y|^2-2|x||y|\cos(k)}.
\end{align*}
Consequently, the quotient \eqref{quo_rpb} can be written as
\begin{align}\label{quo_rpbck}
\sqrt{\frac{1+|x|^2|y|^2-2|x||y|\cos(k)}{|x|^2+|y|^2-2|x||y|\cos(k)+\alpha(1-|x|)(1-|y|)}}.    
\end{align}
By differentiation,
\begin{align*}
&\frac{\partial}{\partial\cos(k)}\left(\frac{1+|x|^2|y|^2-2|x||y|\cos(k)}{|x|^2+|y|^2-2|x||y|\cos(k)+\alpha(1-|x|)(1-|y|)}\right)\\
&=\frac{2|x||y|(1-|x|)(1-|y|)((1+|x|)(1+|y|)-\alpha)}{(|x|^2+|y|^2-2|x||y|\cos(k)+\alpha(1-|x|)(1-|y|))^2}.
\end{align*}
Thus, the quotient \eqref{quo_rpbck} is monotonic with respect to $\cos(k)$ and is at minimum when $\cos(k)=-1$ and at maximum when $\cos(k)=1$ or vice versa, depending on if $\alpha<(1+|x|)(1+|y|)$ or not.

Let us first consider the case $\cos(k)=-1$. Now, the quotient \eqref{quo_rpbck} becomes
\begin{align}\label{quo_rpbc0}
\sqrt{\frac{(1+|x||y|)^2}{(|x|+|y|)^2+\alpha(1-|x|)(1-|y|)}}.    
\end{align}
By differentiation,
\begin{align*}
&\frac{\partial}{\partial|y|}\left(\frac{(1+|x||y|)^2}{(|x|+|y|)^2+\alpha(1-|x|)(1-|y|)}\right)\\
&=\frac{(1-|x|)(1+|x||y|)(\alpha(1-|x||y|+2|x|)-2(1+|x|)(|x|+|y|))}{(|x|+|y|)^2+\alpha(1-|x|)(1-|y|)^2}.
\end{align*}
We see that the only stationary point of the quotient \eqref{quo_rpbc0} with respect to $|y|$ is a maximum. However, the maximum of this quotient \eqref{quo_rpbc0} is the maximum of the quotient \eqref{quo_rpbck} if and only if $\alpha\geq(1+|x|)(1+|y|)$. Because the stationary point fulfills
\begin{align*}
\alpha=\frac{2(1+|x|)(|x|+|y|)}{(1-|x||y|+2|x|)}<(1+|x|)(1+|y|)
\quad\Leftrightarrow\quad
0<(1-|y|)(1+|x||y|),
\end{align*}
it cannot be the maximum of the quotient \eqref{quo_rpbck}. Thus, the quotient \eqref{quo_rpbc0} can offer extreme values of the quotient \eqref{quo_rpbck} only when $|y|\to0^+$ or $|y|\to1^-$. The case $y=0$ was already considered earlier and, if $|y|\to1^-$, the quotient \eqref{quo_rpbc0} approaches 1.  

Let us next consider the case $\cos(k)=1$, where the quotient \eqref{quo_rpbck} is
\begin{align}\label{quo_rpbc1}
\sqrt{\frac{(1-|x||y|)^2}{(|x|-|y|)^2+\alpha(1-|x|)(1-|y|)}}.    
\end{align}
By differentiation,
\begin{align}
&\frac{\partial}{\partial|y|}\left(\frac{(1-|x||y|)^2}{(|x|-|y|)^2+\alpha(1-|x|)(1-|y|)}\right)\nonumber\\
&=\frac{(1-|x|)(1-|x||y|)(\alpha(1+|x||y|-2|x|)+2(1+|x|)(|x|-|y|))}{(|x|-|y|)^2+\alpha(1-|x|)(1-|y|)^2}=0\nonumber\\
&\Leftrightarrow\quad
|y|=\frac{\alpha(1-2|x|)+2|x|(1+|x|)}{2(1+|x|)-\alpha|x|}\label{yroot}.
\end{align}
If $|y|$ is in \eqref{yroot}, the quotient \eqref{quo_rpbc1} is
\begin{align}\label{quo_prbx}
\sqrt{\frac{4(|x|^2+(2-\alpha)|x|+1)^2}{\alpha((4-\alpha)|x|^2+(\alpha^2-6\alpha+8)|x|+4-\alpha)}}.   
\end{align}
Again, by differentiation,
\begin{align*}
\frac{\partial}{\partial|x|}\left(\frac{(|x|^2+(2-\alpha)|x|+1)^2}{(4-\alpha)|x|^2+(\alpha^2-6\alpha+8)|x|+4-\alpha}\right)=\frac{\alpha-2(1+|x|)}{\alpha-4}=0\quad\Leftrightarrow\quad|x|=\frac{\alpha-2}{2}.
\end{align*}
The quotient \eqref{quo_prbx} is 1 at $x=(\alpha-2)/2$. If $|x|\to1^-$, then $|y|\to1^-$ if $|y|$ is as in \eqref{yroot} and the quotient \eqref{quo_prbx} approaches $2/\sqrt{\alpha}$. The quotient \eqref{quo_rpbc1} approaches 1 if $|y|\to1^-$.

Thus, all the potential extreme values of the quotient \eqref{quo_rpb} and their limit values are 1, $1/\sqrt{\alpha}$, $2/\sqrt{\alpha}$, and, if $\alpha<2$, $2/\sqrt{\alpha(4-\alpha)}$. Note that $2/\sqrt{\alpha(4-\alpha)}$ is never an extreme value of this quotient because it is obtained only if $0<\alpha<2$ and the inequality $1/\sqrt{\alpha}<2/\sqrt{\alpha(4-\alpha)}<2/\sqrt{\alpha}$ holds for all $0<\alpha<3$. The inequality of the theorem follows and, since the values are either extreme values of the quotient \eqref{quo_rpb} or their limit values, there are no better constants in terms of $\alpha$.
\end{proof}

\begin{corollary}
For all $x,y\in\B^n$ and $\alpha>0$ and any conformal mapping $f:\B^n\to\B^n=f(\B^n)$,
\begin{align*}
\min\left\{\frac{\sqrt{\alpha}}{2},\frac{1}{2},\frac{1}{\sqrt{\alpha}}\right\}p^\alpha_{\B^n}(x,y)\leq 
p^\alpha_{\B^n}(f(x),f(y))\leq\max\left\{\frac{2}{\sqrt{\alpha}},2,\sqrt{\alpha}\right\}p^\alpha_{\B^n}(x,y).    
\end{align*}
\end{corollary}
\begin{proof}
By Theorem \ref{thm_prhoB} and the conformal invariance of the hyperbolic metric,
\begin{align*}
p^\alpha_{\B^n}(f(x),f(y))
&\leq\max\left\{1,\frac{2}{\sqrt{\alpha}}\right\}{\rm th}\frac{\rho_{\B^n}(f(x),f(y))}{2}
=\max\left\{1,\frac{2}{\sqrt{\alpha}}\right\}{\rm th}\frac{\rho_{\B^n}(x,y)}{2}\\
&\leq\max\left\{1,\frac{2}{\sqrt{\alpha}}\right\}\frac{1}{\min\left\{1,1/\sqrt{\alpha}\right\}}p^\alpha_{\B^n}(x,y)
=\max\left\{\frac{2}{\sqrt{\alpha}},2,\sqrt{\alpha}\right\}p^\alpha_{\B^n}(x,y)
\end{align*}
and, since the inverse mapping $f^{-1}$ of any conformal mapping is another conformal mapping, the first part of inequality follows directly from this.
\end{proof}

Consider the M\"obius transformation $T_a:\B^2\to\B^2$, defined as $T_a(z)=(z-a)\slash(1-\overline{a}z)$. It has been observed that the Lipschitz constant of this mapping seems to be $1+|a|$ for several intrinsic metrics and quasi-metrics defined in the unit disk, including the triangular ratio metric \cite[Conj. 1.6, p. 684]{chkv}, the $j^*$-metric \cite{sch}, the $t$-metric \cite[Conj. 4.4]{inm}, the point pair function \cite{sch}, and the Barrlund metric \cite[Conj. 4.3, p. 25]{fmv}. Computer tests suggest that this also holds for the generalized point pair function, regardless of the value of $\alpha>0$.  

\begin{conjecture}
For all $x,y,a\in\B^2$ and $\alpha>0$,
\begin{align*}
\frac{1}{1+|a|}p^\alpha_{\B^2}(x,y)\leq p^\alpha_{\B^2}(T_a(x),T_a(y))\leq(1+|a|)p^\alpha_{\B^2}(x,y)    
\end{align*}
\end{conjecture}

\begin{definition}\cite[p. 288-289]{hkv}
Let $G\subset\R^n$ be a domain. See \cite[Def. 9.1, p. 149]{hkv} for a definition of a function that is absolute continuous on lines, abbreviated as ACL. Denote the derivative of $f$ at $x$ by $f'(x)$ and the Jacobian determinant of $f$ at $x$ by $J_f(x)$. A mapping $f:G\to\R^n$ is quasiregular if it is $\text{ACL}^n$ and there is constant $K\geq1$ such that 
\begin{align}\label{ine_qrf}
|f'(x)|^n\leq KJ_f(x),\quad |f'(x)|=\max_{|h|=1}|f'(x)h|    
\end{align}
holds almost everywhere (a.e.) in $G$. If $f$ is quasiregular, then the smallest $K\geq1$ with which \eqref{ine_qrf} holds is the outer dilatation of $f$, denoted by $K_O(f)$, and the smallest $K\geq1$ such that the inequality
\begin{align*}
J_f(x)\leq Kl(f'(x))^n,\quad \min_{|h|=1}|f'(x)h|    
\end{align*}
holds a.e. in $G$ is the inner dilatation of $f$, denoted by $K_I(f)$. A quasiregular mapping $f$ is $K$-quasiregular if 
\begin{align*}
\max\{K_O(f),K_I(f)\}\leq K.    
\end{align*}    
\end{definition}

By \cite[(9.5), p. 157, \& (9.6), p. 158]{hkv}, define a constant 
\begin{align*}\label{q_ldn}
\log\lambda_n=\lim_{t\to\infty}\left(\left(\frac{\gamma_n(t)}{\omega_{n-1}}\right)^{n-1}-\log(t)\right),    
\end{align*}
where $\gamma_n$ is the Gr\"otzsch capacity defined as in \cite[(7.17), p. 121]{hkv}. By \cite[(7.18), p. 122]{hkv}, in the two-dimensional case
\begin{align*}
\gamma_2(t)=\frac{4\K(1/t)}{\K(\sqrt{1-1/t^2})},  
\end{align*}
where $\K$ is a complete elliptic integral of the first kind. This integral is defined as
\begin{align*}
\K(r)=\int^1_0\frac{1}{\sqrt{(1-x^2)(1-r^2x^2)}}dx,\quad 0<r<1,    
\end{align*}
and can be computed with ready functions in many programming languages.

\begin{theorem}\label{thm_s} \cite[Thm 16.2(1), p. 300]{hkv}
If $G,G'\in\{\uhp^n,\B^n\}$ and $f:G\to G'$ is a non-constant $K$-quasiregular mapping with $f(G)\subset G'$, then
\begin{align*}
{\rm th}\frac{\rho_{G'}(f(x),f(y))}{2}
\leq\lambda_n^{1-c}\left({\rm th}\frac{\rho_G(x,y)}{2}\right)^c,
\end{align*}
where $c=K_I(f)^{1/(1-n)}$.
\end{theorem}

\begin{corollary}
If $f:\B^n\to\B^n=f(\B^n)$ is a non-constant $K$-quasiregular mapping, then for all $x,y\in\B^n$ and $\alpha>0$
\begin{align*}
p^\alpha_{\B^n}(f(x),f(y))\leq\lambda_n^{1-c}\max\left\{1,\frac{2}{\sqrt{\alpha}},(\sqrt{\alpha})^c,2(\sqrt{\alpha})^{c-1}\right\}p^\alpha_{\B^n}(x,y)^c,    
\end{align*}
where $c=K_I(f)^{1/(1-n)}$.
\end{corollary}
\begin{proof}
By Theorems \ref{thm_prhoB} and \ref{thm_s}, 
\begin{align*}
p^\alpha_{\B^n}(f(x),f(y))
&\leq\max\left\{1,\frac{2}{\sqrt{\alpha}}\right\}{\rm th}\frac{\rho_{\B^n}(f(x),f(y))}{2}
\leq\lambda_n^{1-c}\max\left\{1,\frac{2}{\sqrt{\alpha}}\right\}\left({\rm th}\frac{\rho_{\B^n}(x,y)}{2}\right)^c\\
&\leq\lambda_n^{1-c}\max\left\{1,\frac{2}{\sqrt{\alpha}}\right\}\left(\frac{1}{\min\left\{1,1/\sqrt{\alpha}\right\}}p^\alpha_{\B^n}(x,y)\right)^c\\
&=\lambda_n^{1-c}\max\left\{1,\frac{2}{\sqrt{\alpha}},(\sqrt{\alpha})^c,2(\sqrt{\alpha})^{c-1}\right\}p^\alpha_{\B^n}(x,y)^c.
\end{align*}
\end{proof}

\def\cprime{$'$} \def\cprime{$'$} \def\cprime{$'$}
\providecommand{\bysame}{\leavevmode\hbox to3em{\hrulefill}\thinspace}
\providecommand{\MR}{\relax\ifhmode\unskip\space\fi MR }
\providecommand{\MRhref}[2]{%
  \href{http://www.ams.org/mathscinet-getitem?mr=#1}{#2}
}
\providecommand{\href}[2]{#2}

\end{document}